\documentclass[12pt,a4paper]{amsart}
\usepackage[matrix,arrow]{xy}
\usepackage{amssymb}
\usepackage{amsmath}
\usepackage{amsfonts}
\usepackage{epsfig}
\usepackage{color}
\usepackage{epstopdf}
\usepackage{graphicx}
\usepackage{amsthm}
\usepackage{enumerate}
\usepackage[mathscr]{eucal}
\usepackage{lipsum} 
\usepackage{verbatim}

\usepackage[bookmarks=true,hyperindex,pdftex,colorlinks,citecolor=blue,
linkcolor=blue,urlcolor=blue]{hyperref}
\usepackage{tikz}
\usetikzlibrary{matrix,arrows.meta}

\theoremstyle{plain}
\newtheorem{theorem}{Theorem}[section]
\newtheorem{cor}[theorem]{Corollary}
\newtheorem{prop}[theorem]{Proposition}
\newtheorem{lemma}[theorem]{Lemma}

\theoremstyle{definition}

\newtheorem{remark}[theorem]{Remark}
\newtheorem{definition}[theorem]{Definition}

\newcommand{\Lin}{\mathcal{L}}
\newcommand{\K}{\mathcal{K}}

\newcommand{\eps}{\varepsilon}
\newcommand{\iten}{\widehat{\otimes}_{\eps}}

\DeclareMathOperator{\dist}{dist}

\DeclareMathOperator{\diam}{diam}

\DeclareMathOperator{\ACK}{ACK}

\DeclareMathOperator{\Ext}{Ext}

\DeclareMathOperator{\NA}{NA}

\renewcommand{\subset}{\subseteq}

\newcommand{\pten}{\widehat{\otimes}_\pi}

\title[A generalized ACK structure and norm attaining operators]
{A generalized ACK structure and the denseness of norm attaining operators}

\author[Choi]{Geunsu Choi}
\address[Choi]{Department of Mathematics Education, Dongguk University, Seoul 04620, Republic of Korea \newline
\href{http://orcid.org/0000-0002-4321-1524}{ORCID: \texttt{0000-0002-4321-1524} }}
\email{\texttt{chlrmstn90@gmail.com}}

\author[Jung]{Mingu Jung}
\address[Jung]{School of Mathematics, Korea Institute for Advanced Study, Seoul 02455, Republic of Korea \newline
\href{http://orcid.org/0000-0003-2240-2855}{ORCID: \texttt{0000-0003-2240-2855} }}
\email{\texttt{jmingoo@kias.re.kr}}
\urladdr{\url{https://clemg.blog/}}


\date{\today}                      

\begin{document}

\begin{abstract}
Inspired by the recent work of Cascales et al., we introduce a generalized concept of ACK structure on Banach spaces. 
Using this property, which we call by the quasi-ACK structure, we are able to extend known universal properties on range spaces concerning the density of norm attaining operators. We provide sufficient conditions for quasi-ACK structure of spaces and results on the stability of quasi-ACK structure. As a consequence, we present new examples satisfying (Lindenstrauss) property B$^k$, which have not been known previously. We also prove that property B$^k$ is stable under injective tensor products in certain cases. 
Moreover, ACK structure of some Banach spaces of vector-valued holomorphic functions is also discussed, leading to new examples of universal BPB range spaces for certain operator ideals.
\end{abstract}

\maketitle

\section{Introduction}

Lindenstrauss \cite{Lin} established the concept of approximation on operators between Banach spaces by those which attain their norm in his contribution to the theory of norm attaining operators. Namely, he gave a number of criteria on Banach spaces for which the set of norm attaining operators is dense in the space of bounded linear operators. In particular, so called \emph{\textup{(}Lindenstrauss\textup{)} property \textup{B}} on a range space, the property which states that the set of norm attaining operators is dense for an arbitrary domain space, was considered in his seminal work and a sufficient condition was studied by giving certain constraints on the range space. This property plays a fundamental role in the field of norm attaining theory in view of geometric aspect of Banach spaces, and one of the significant branches along this line is about \emph{property \textup{B}$^k$}, a parallel study restricted to the compact operators. We invite the readers to \cite{AAP2,CGK,DGMM,JW1,M2014} for the study of norm attaining compact operators, and also to a survey \cite{M2016}. 
It was recently observed in \cite{CGKS2018} that if a Banach space satisfies a specific structural property, so-called \emph{$\ACK$ structure}, then it can be said to be more than just a universal range space of denseness. The primary goal of this article is to generalize their concept while retaining the essential properties of denseness, allowing us to provide new constructive examples.

Let us introduce the entire plot of contents once more in detail this time. For (real or complex) Banach spaces $X$ and $Y$, we will use the notation $S_X, B_X, X^*$ for the unit sphere, the closed unit ball and the dual space of $X$, respectively, $\Ext (B_X)$ for the set of extreme points of $B_X$, $\Lin (X,Y), \mathcal{K}(X,Y)$ for the space of bounded linear operators and of compact linear operators from $X$ to $Y$, respectively, and $\NA (X, Y)$ for the set of all norm attaining operators in $\Lin (X,Y)$. For simplicity, we abbreviate $\Lin (X,X)$ by $\Lin (X)$.

The \emph{Bishop-Phelps-Bollob\'as property} (for short, \textup{BPBp}), introduced in 2008 \cite{AAGM} as a vector-valued version of the classical Bishop-Phelps-Bollob\'as theorem \cite{B}, has attracted considerable attention of many authors over the years as it extends the classical denseness of norm attaining operators in quantitative ways. A pair $(X,Y)$ of Banach spaces is said to have the \textup{BPBp} for operators if for every $\eps >0$, there exists $\eta(\eps) >0$ such that for every $T \in \Lin (X,Y)$ with $\|T\|=1$, if $x_0 \in S_X$ satisfies $\|T(x_0)\| > 1 -\eta (\eps)$, then there exist a point $u_0 \in S_X$ and an operator $S \in \Lin (X,Y)$ satisfying that $\|S(u_0)\| = \|S\|=1$, $\|u_0 - x_0 \| < \eps$, and $\|S - T\|<\eps$.

In 2018, Cascales, Guirao, Kadets, and Soloviova \cite{CGKS2018} introduced a new structure of Banach spaces, called {$\ACK_\rho$ structure} by extracting the structural properties of $C(K)$-spaces and uniform algebras. In order to explain their result, we need to recall some definition: a function $f$ from a topological space $\mathcal{T}$ to a metric space $M$ is said to be \emph{openly fragmented} if for every nonempty open subset $U \subseteq \mathcal{T}$ and every $\eps >0$ there exists a nonempty open subset $V \subseteq U$ with $\text{diam} (f(V)) < \eps$. Given Banach spaces $X, Y$ and a set $\Gamma \subseteq Y^*$, an operator $T \in \Lin (X,Y)$ is said to be $\Gamma$-\emph{flat} if $T^* \vert_{\Gamma} : (\Gamma, w^*) \rightarrow X^*$ is openly fragmented, and a subspace $\mathcal{I} \subseteq \mathcal{L}(X,Y)$ is said to be a \emph{$\Gamma$-flat ideal} if every element of $\mathcal{I}$ is $\Gamma$-flat, $\mathcal{I}$ contains all finite rank operators and $FT \in \mathcal{I}$ for every $T \in \mathcal{I}$ and $F \in \Lin(Y)$.
It is known that if $T \in \Lin (X,Y)$ is an Asplund operator, then the set $T^*( B_{Y^*} )$ is a $w^*$-compact subset of $(X^*, w^*)$ and it is fragmented by the norm \cite{Stegall}. That is, each Asplund operator (in particular, compact or weakly compact operator) $T \in \Lin (X,Y)$ is $\Gamma$-flat for every subset $\Gamma \subseteq B_{Y^*}$ (see also \cite[Example A]{CGKS2018}).

A Banach space $X$ has \emph{$\ACK_\rho$ structure with parameter $\rho \in [0,1)$} whenever there exists a $1$-norming set $\Gamma \subseteq B_{X^*}$ such that for every $\eps >0$ and every nonempty relatively $w^*$-open subset $U \subseteq \Gamma$, there exist a nonempty subset $V \subseteq U$, points $x_1^* \in V$, $e \in S_X$ and an operator $F \in \Lin (X)$ with the following properties:
\begin{enumerate}
\setlength\itemsep{0.3em}
\item[\textup{(i)}] $\|F(e)\| = \|F \| = 1$;
\item[\textup{(ii)}] $x_1^*(F(e))=1$,
\item[\textup{(iii)}] $F^*(x_1^*)=x_1^*$,
\item[\textup{(iv)}] If we let $V_1 := \{ x^* \in \Gamma : \|F^*(x^*)\| + (1-\eps) \|x^* - F^*(x^*)\| \leq 1 \}$, then $|x^*(F(e))| \leq \rho$ for any $x^* \in \Gamma \setminus V_1$,
\item[\textup{(v)}] $\dist (F^*(x^*), \operatorname{aco}\{0,V\}) < \eps$ for all $x^* \in \Gamma$,
\item[\textup{(vi)}] $|v^*(e)-1| \leq \eps$ for all $v^* \in V$.
\end{enumerate} 
If the set $V_1$ given in (iv) is empty, then we write simply by $X$ has \emph{$\ACK$ structure}. 

Cascales et al. showed in \cite{CGKS2018} that uniform algebras and Banach spaces with property $\beta$ have $\ACK_\rho$ structure, and presented several stability results concerning $\ACK_\rho$ structure. Moreover, they proved that a pair \emph{$(X,Y)$ has the \textup{BPBp} for operators in a $\Gamma$-flat ideal $\mathcal{I}$} provided that $Y$ has $\ACK_\rho$ structure with the corresponding $1$-norming set $\Gamma \subseteq B_{X^*}$, which unifies and extends several results in \cite{AAGM, ABGKM, ACK, CGK}. The notion of the \textup{BPBp} for operators in $\mathcal{I} \subseteq \mathcal{L}(X,Y)$ is an analogue of the original BPBp with restricting $T, S \in \mathcal{I}$. In other words, if for every $\eps >0$, there is $\eta(\eps) >0$ such that for every $T \in \mathcal{I}$ with $\|T\|=1$, if $x_0 \in S_X$ satisfies $\|T(x_0)\| > 1 -\eta (\eps)$, then there exist a point $u_0 \in S_X$ and an operator $S \in \mathcal{I}$ satisfying that $\|S(u_0)\| = \|S\|=1$, $\|u_0 - x_0 \| < \eps$, and $\|S - T\|<\eps$, then we say that the pair $(X,Y)$ has the BPBp for operators in $\mathcal{I}$.

In this paper, we introduce a new structural Banach space property which we call by \emph{quasi-$\ACK$ structure}. We show that this notion is more general than $\ACK_\rho$ structure, while the same property on a range space still implies the denseness of norm attaining operators in some $\Gamma$-flat ideals for an arbitrary domain space. Let us mention that the idea behind this notion comes from the paper of Acosta, Aguirre and Pay\'a \cite{AAP2}, where {property quasi-$\beta$} was introduced as a property which is (strictly) weaker than so called property $\beta$. Recall that a Banach space $Y$ has \emph{property quasi-$\beta$} if there exist subsets $A = \{ y_\lambda^*: \lambda \in \Lambda\} \subseteq S_{Y^*}$, $\{y_\lambda : \lambda \in \Lambda\} \subseteq S_Y$ and a real-valued function $\rho$ on $A$ satisfying 
\begin{enumerate}
\setlength\itemsep{0.3em}
\item[\textup{(i)}] $y_\lambda^* (y_\lambda) = 1$ for every $\lambda \in \Lambda$,
\item[\textup{(ii)}] $|y_\mu^* (y_\lambda)| \leq \rho (y_\lambda^*) < 1$ whenever $\lambda, \mu \in \Lambda$ with $\lambda \neq \mu$,
\item[\textup{(iii)}] for every $e^* \in \Ext (B_{Y^*})$, there is a subset $A_{e^*} \subseteq A$ such that $e^* \in \overline{\mathbb{T} A_{e^*}}^{w^*}$ and $\sup \{ \rho (y^*):y^* \in A_{e^*}\} < 1$, where $\mathbb{T} = \{t \in \mathbb{K} : |t|=1 \}$.
\end{enumerate} 
If the function $\rho : A \rightarrow [0,1)$ is bounded above by some $\sigma \in [0,1)$ and each $A_{e^*}$ can be chosen to be $A$ (equivalently, $A$ is norming), this definition turns out to coincide with \emph{property $\beta$}, which was originally introduced by Lindenstrauss. Actually in his seminal paper \cite{Lin}, a Banach space $Y$ is defined to have \emph{property \textup{B}} if the set $\NA (X,Y)$ is dense in $\Lin (X,Y)$ for every Banach space $X$, and it is proved that property $\beta$ is a sufficient condition for property \textup{B}. 
Typical examples of Banach spaces which have property $\beta$ are polyhedral finite dimensional Banach spaces and closed subspaces of $\ell_\infty$ containing the canonical copy of $c_0$. 
The problem of Banach spaces with property \textup{B} is later considered in the context of compact operators. Following \cite{M2016}, we say that $Y$ has property $\textup{B}^k$ if $\NA(X,Y) \cap \K (X,Y)$ is dense in $\K (X,Y)$ for every domain space $X$. The space $C[0,1]$ for instance has property \textup{B}$^k$ but fails property \textup{B} \cite{S}, and it was proved in \cite{M2014} that there is a Banach space which fails property \textup{B}$^k$. However, it is still open whether property \textup{B} implies property \textup{B}$^k$ or not.

The aim of this paper is to show that the notion of quasi-$\ACK$ structure implies property \textup{B} for some $\Gamma$-flat ideal $\mathcal{I}$, that is, every operator in $\Lin (X,Y)$ belonging to $\mathcal{I}$ can be approximated by norm attaining operators in the same ideal for every Banach space $X$ whenever a Banach space $Y$ has quasi-$\ACK$ structure. Moreover, we observe that quasi-$\ACK$ structure is a non-trivial extension of $\ACK_\rho$ structure by showing the existence of a Banach space having quasi-$\ACK$ structure but lacks $\ACK_\rho$ structure and property quasi-$\beta$. The existence of such a space yields a new example of Banach spaces that satisfies property \textup{B}$^k$. Next, we focus our attention on the inheritance of quasi-$\ACK$ structure from a Banach space $Z$ to a vector-valued function space with target space $Z$. We obtain in addition some new results on $\ACK_\rho$ structure of some spaces of holomorphic functions between Banach spaces. Furthermore, we prove that property $\beta$, property quasi-$\beta$, and quasi-$\ACK$ structure are stable under injective tensor products. As a related result, we also show that property \textup{B}$^k$ of a reflexive space implies property \textup{B}$^k$ of the injective tensor product between itself and an $L_1$-predual space. Finally, we give a direct proof of quasi-$\ACK$ structure of the $c_0$-sum of spaces having quasi-$\ACK$ structure, which actually yields $\ACK_\rho$ structure of an intermediate space between the $c_0$-sum and $\ell_\infty$-sum of a family $(X_i)_{i\in I}$ of Banach spaces having $\ACK_\rho$ structure. 

\section{Basic properties of quasi-ACK structure}

Let us present the promised definition of quasi-$\ACK$ structure.

\begin{definition}\label{def:qACK}
Let $X$ be a Banach space. $X$ is said to have \emph{quasi-$\ACK$ structure} if there exist a $1$-norming set $\Gamma \subseteq B_{X^*}$ and a function $\rho: \Gamma \to [0,1)$ such that for any $e^* \in \operatorname{Ext}(B_{X^*})$, there exists $\Gamma_{e^*} \subseteq \Gamma$ satisfying that
\begin{enumerate}
\setlength\itemsep{0.3em}
\item[\textup{(i)}] $e^* \in \overline{\mathbb{T} \Gamma_{e^*}}^{w^*}$,
\item[\textup{(ii)}] $\sup_{x^* \in \Gamma_{e^*}} \rho(x^*) <1$,
\item[\textup{(iii)}] for every $\eps>0$ and a nonempty relatively $w^*$-open subset $U$ of $\Gamma_{e^*}$, there are a nonempty subset $V \subseteq U$, $x_1^* \in V$, $e \in S_X$, $F \in \Lin(X)$ such that 
\item[\textup{(i)'}] $\|F(e)\|=\|F\|=1$,
\item[\textup{(ii)'}] $x_1^*(F(e))=1$,
\item[\textup{(iii)'}] $F^*(x_1^*)=x_1^*$,
\item[\textup{(iv)'}] If we let $V_1 := \{ x^* \in \Gamma : \|F^*(x^*)\| + (1-\eps) \|x^* - F^*(x^*)\| \leq 1 \}$, then $|x^*(F(e))| \leq \rho(x_1^*)$ for any $x^* \in \Gamma \setminus V_1$,
\item[\textup{(v)'}] $\dist (F^*(x^*), \operatorname{aco}\{0,V\}) < \eps$ for all $x^* \in \Gamma$,
\item[\textup{(vi)'}] $|v^*(e)-1| \leq \eps$ for all $v^* \in V$.
\end{enumerate}
\end{definition}

\begin{remark}\label{rem:basic} Let $X$ be a Banach space. 
\begin{enumerate}
\setlength\itemsep{0.3em}
\item[\textup{(a)}] If $X$ has $\ACK_\rho$ structure for some $\rho \in [0,1)$, then $X$ has quasi-$\ACK$ structure.
\item[\textup{(b)}] If $X$ has property quasi-$\beta$, then $X$ has quasi-$\ACK$ structure.
\item[\textup{(c)}] There is a Banach space $X$ with quasi-$\ACK$ structure but without $\ACK_\rho$ structure for any $\rho \in [0,1)$. 
\end{enumerate} 
\end{remark}

\begin{proof}
(a). Let $\Gamma$ be a $1$-norming set from the definition of $\ACK_\rho$ structure. 
Consider the value $\rho$ as a constant function from $\Gamma$ to $[0,1)$ and take $\Gamma_{e^*}=\Gamma$ for any $e^* \in \operatorname{Ext}(B_{X^*})$.

(b). Take subsets $A = \{x_\alpha^*: \alpha \in \Lambda\} \subseteq S_{X^*}$, $\{x_\alpha : \alpha \in \Lambda\} \subseteq S_X$ and a real-valued function $\rho$ on $A$ satisfying the conditions in the definition of property quasi-$\beta$. Then $X$ has quasi-$\ACK$ structure with respect to the $1$-norming set $\Gamma:= A$ and the above function $\rho$. Indeed, for $e^* \in \operatorname{Ext}(B_{X^*})$, let $\Gamma_{e^*}$ be the corresponding subset of $\Gamma$ in the definition of property quasi-$\beta$. Given any $\eps>0$ and a nonempty relatively $w^*$-open set $U \subseteq \Gamma_{e^*}$, fix $x_{\alpha_0}^* \in U$ and we may take
$$
V := \{x_{\alpha_0}^*\} \subseteq U, \quad x_1^* := x_{\alpha_0}^*, \quad e := x_{\alpha_0} \quad \text{and} \quad F(x) := x_{\alpha_0}^*(x) x_{\alpha_0} \in \Lin(X).
$$
It is routine to show that (i)-(iii) hold.

(c). Let $X$ be the Banach space considered in \cite[Example 4.1]{ACKLM}, which has property quasi-$\beta$ and satisfies that $(\ell_1^2, X)$ fails to have the \textup{BPBp} for operators. It follows from \cite[Theorem 3.4]{CGKS2018} that $X$ fails to have $\ACK_\rho$ structure for any $\rho \in [0,1)$. 
\end{proof}

As mentioned in Introduction, we shall observe later that there exists a Banach space $X$ which is of quasi-$\ACK$ structure but fails to have not only $\ACK_\rho$ structure for any $\rho \in [0,1)$ but also property quasi-$\beta$ (see Proposition \ref{prop:qACKbut}). On the one hand, the following theorem shows that the construction of quasi-$\ACK$ structure is very natural compared to the defintion of property quasi-$\beta$ and does not harm the essential property on the denseness of norm attaining operators.

\begin{theorem}\label{thm:denseness}
Let $X$ and $Y$ be Banach spaces such that $Y$ has quasi-$\ACK$ structure with a $1$-norming set $\Gamma \subseteq B_{Y^*}$ and $\rho : \Gamma \rightarrow [0,1)$. Let $T \in \Lin(X,Y)$ be a $\Gamma$-flat operator such that $T$ is $\Gamma_0$-flat for every $\Gamma_0 \subseteq \Gamma$. Then, $T$ can be approximated by norm attaining operators.
\end{theorem}

\begin{proof}
Let $\eps>0$ be given and assume that $\|T\|=1$. By Johannesen \cite[Theorem 5.8]{Lima}, $T^* \in \Lin(Y^*,X^*)$ attains its norm at some $e^* \in \operatorname{Ext}(B_{Y^*})$. Let $\Gamma_{e^*} \subset \Gamma$ be the corresponding set of quasi-$\ACK$ structure of $Y$ satisfying (i)-(iii). From that $\|T^*(e^*)\|=1$, there exist $y_0^* \in \Gamma_{e^*}$ and $x_0 \in S_X$ such that $|y_0^*(T(x_0))| > 1-\eps$ by (i). If we fix
$$
U_0 := \{y^* \in Y^* : |y^*(T(x_0))| >1-\eps\},
$$
then $U_0$ is $w^*$-open and $U_0 \cap \Gamma_{e^*} \neq \emptyset$. Since $T$ is $\Gamma_{e^*}$-flat from the hypothesis, there exists $w^*$-open $U_r \subseteq U_0$ with $U_r \cap \Gamma_{e^*} \neq \emptyset$ such that $\diam(T^*(U_r \cap \Gamma_{e^*})) < \eps$. Choose any $y_1^* \in U_r \cap \Gamma_{e^*}$ and let $x_1^* := T^*(y_1^*)$. By the Bishop-Phelps theorem, we may find $x_r^* \in \NA(X,\mathbb{K})$ with $\|x_r^*\|=1$ such that $\|x_r^*-x_1^*\|<\eps$. Therefore, we can deduce that
\begin{equation}\label{equation:nbh}
\|T^*(z^*) - x_r^*\| \leq \|T^*(z^*) - T^*(y_1^*)\| + \|x_1^* - x_r^*\| < 2\eps \quad \text{for all } z^* \in U_r \cap \Gamma_{e^*}.
\end{equation}
Let us say that $x_r^*$ attains its norm at $x_r \in S_X$.

Now, by quasi-$\ACK$ structure of $Y$, applying to $\eps>0$ and $U := U_r \cap \Gamma_{e^*}$, we are able to find a nonempty $V \subseteq U$, $y_2^* \in V$, $e \in S_Y$ and $F \in \Lin(Y)$ satisfying (i)'-(vi)'. Put $V_1 = \{x^* \in \Gamma : \|F^*(x^*)\| + (1-\eps) \|x^* - F^*(x^*)\| \leq 1 \}$. 
Define $S \in \Lin(X,Y)$ by
$$
S(x) := x_r^*(x) F(e) + (1-\tilde{\eps}) (I_Y - F) T(x) \qquad \text{for } x \in X,
$$
where $\tilde{\eps} := \frac{5\eps}{1-r_{e^*}+5\eps}$. Here, $r_{e^*} := \sup_{y^* \in \Gamma_{e^*}} \rho(y^*)<1$ from (ii). Note first that $\|S\| = \sup_{y^* \in \Gamma} \|S^*(y^*)\|$. If $y^* \in V_1$, then
\begin{align*}
\|S^*(y^*)\| &= \|y^*(F(e)) x_r^* + (1-\tilde{\eps}) T^*(y^* - F^*(y^*))\| \\
&\leq \|F^*(y^*)\| + (1-\tilde{\eps}) \|y^*-F^*(y^*)\| \leq 1.
\end{align*}
According to (v)' and (vi)', we have for any $y^* \in \Gamma$ an estimate $\|F^*(y^*)-v^*\|<\eps$ where $v^* = \sum_{k=1}^n \lambda_k v_k^*$ with $\{v_k^*\}_{k=1}^n \subseteq V$ and $\sum_{k=1}^n |\lambda_k| \leq 1$. Hence we have from \eqref{equation:nbh} and (vi)' that
$$
\|v^*(e) x_r^* - T^*(v^*)\| \leq |v^*(e)-1| + \left| \sum_{k=1}^n \lambda_k (x_r^* - T(v_k^*)) \right|  < \eps + 2\eps \sum_{k=1}^n |\lambda_k|  \leq 3\eps.
$$
So it follows for $y^* \in \Gamma \setminus V_1$ that
\begin{align*}
\|S^*(y^*)\| & \leq \tilde{\eps} |y^*(F(e))| + (1-\tilde{\eps}) \|T^*(y^*)\| + (1-\tilde{\eps}) \|(F^*(y^*))(e)x_r^* - T^*F^*(y^*)\| \\
&\leq \tilde{\eps} r_{e^*} + (1-\tilde{\eps}) +2\eps(1-\tilde{\eps}) + 3\eps(1-\tilde{\eps}) \\
&= \tilde{\eps} r_{e^*} + (1-\tilde{\eps}) (1+5\eps) \\
&\leq 1
\end{align*}
from the choice of $\tilde{\eps}$ since $y_2^* \in \Gamma_{e^*}$ and by (iv)'. These results give that $\|S\| \leq 1$. On the other hand, notice from (ii)' and (iii)' that
$$
1 = |x_r^*(x_r)| = |y_2^*(S(x_r))| \leq \|S(x_r)\| \leq \|S\| \leq 1,
$$
hence $S$ attains its norm at $x_r$. Finally, we have
\begin{align*}
\|S-T\| &= \sup_{y^* \in \Gamma} \|S^*(y^*) - T^*(y^*)\| \\
&\leq \sup_{y^* \in \Gamma} \|y^*(F(e))x_r^* - T^*F^*(y^*)\| + 2\tilde{\eps} \leq 5\eps + 2\tilde{\eps}
\end{align*}
arguing similarly as above. This concludes the proof.
\end{proof}

Observe that the operator $S$ in the proof of Theorem \ref{thm:denseness} is $\Gamma_0$-flat for every $\Gamma_0 \subseteq \Gamma$. As mentioned in Introduction, each Asplund operator $T \in \Lin (X,Y)$ is $\Gamma$-flat for every subset $\Gamma \subseteq B_{Y^*}$. Thus, one of the direct consequences of Theorem \ref{thm:denseness} is that a Banach space $Y$ with quasi-$\ACK$ structure satisfies property \textup{B}$^k$, i.e., $\NA(X, Y) \cap \mathcal{K} (X, Y) $ is dense in $\mathcal{K} (X,Y)$ for every Banach space $X$.

\begin{cor}\label{cor:denseness}
Let $X$ and $Y$ be Banach spaces. Let $\mathcal{I}$ be either an ideal of compact, weakly compact, or Asplund operators from $X$ into $Y$. 
\begin{enumerate}
\setlength\itemsep{0.3em}
\item[\textup{(a)}]
If $Y$ has quasi-$\ACK$ structure, then every operator $T \in \mathcal{I}$ can be approximated by norm attaining operators which are chosen from $\mathcal{I}$ as well.
\item[\textup{(b)}] \textup{(}Acosta, Aguirre and Pay\'a \cite{AAP2}\textup{)}
If $Y$ has property quasi-$\beta$, then $\NA(X,Y)$ is dense in $\Lin(X,Y)$.
\end{enumerate}
\end{cor}

\begin{proof}
(a). As mentioned above, the operator $S$ in the proof of Theorem \ref{thm:denseness} is in the ideal $\mathcal{I}$ whenever $T$ belongs to $\mathcal{I}$.

(b). 
Note from Remark \ref{rem:basic} that $Y$ has quasi-$\ACK$ structure. Since $(\Gamma, w^*)$ is a discrete topological space, every operator $T \in \Lin(X,Y)$ is $\Gamma_0$-flat for any $\Gamma_0 \subseteq \Gamma$ \cite[Example C]{CGKS2018}. 
\end{proof}

\section{Function spaces with quasi-ACK structure}

\subsection{Results on quasi-ACK structure}
Given a compact Hausdorff space $K$ and a (unital) uniform algebra $A$ on $K$, the Choquet boundary and Shilov boundary for $A$ will be denoted by $\chi A$ and $\partial A$, respectively. We refer the reader to \cite{Dales, Stout} for background on uniform algebras and related notions.

\begin{theorem}\label{thm:unifalgebra}
Let $A$ be a uniform algebra on a compact Hausdorff space $K$ and $Z$ be a Banach space with quasi-$\ACK$ structure. 
Suppose that a Banach space $X \subseteq C(\partial A, Z)$ satisfies the following properties: 
\begin{enumerate}
\setlength\itemsep{0.3em}
\item[\textup{(i)}] For every $x \in X$ and $f \in A$, the function $t \mapsto f(t) x$ belongs to $X$.
\item[\textup{(ii)}] $X$ contains all functions of the form $f \otimes z$ for $f \in A$ and $z \in Z$. 
\item[\textup{(iii)}] $T\circ x \in X$ for every $x \in X$ and $T \in \mathcal{L} (Z)$. 
\end{enumerate} 
Then $X$ has quasi-$\ACK$ structure. 
\end{theorem} 

\begin{proof}
Let $\Gamma_Z$ be a $1$-norming set and $\rho_Z : \Gamma_Z \rightarrow [0,1)$ be a function which witness quasi-$\ACK$ structure of $Z$. Define the set $\Gamma := \{ \delta_t \otimes z^* : t \in \partial A, z^* \in \Gamma_Z \}$ and the function $\rho : \Gamma \rightarrow [0,1)$ by $\rho (\delta_t \otimes z^*):= \rho_Z (z^*)$ for every $t \in \partial A$ and $z^* \in \Gamma_Z$. It is clear that $\Gamma \subseteq B_{X^*}$ is a $1$-norming set for $X$. 

Let $e^* \in \Ext (B_{X^*} )$ be given. By \cite[Corollary 3.4]{BD}, we have that $e^* = \delta_{t_0} \otimes z_0^*$ for some $t_0 \in \partial A$ and $z_0^* \in \Ext (B_{Z^*})$. Let $\Gamma_{z_0^*} \subseteq \Gamma_Z$ be a set satisfying the properties (i)-(iii) in Definition \ref{def:qACK}. Define the set $\Gamma_{e^*} := \{ \delta_t \otimes z^* : t \in \chi A, z^* \in \Gamma_{z_0^*} \}$. It is clear that $\Gamma_{e^*} \subseteq \Gamma$. Moreover, $e^* = \delta_{t_0} \otimes z_0^*$ belongs to $\overline{\mathbb{T} \Gamma_{e^*} }^{w^*}$. Indeed, if we take a net $(t_\alpha) \subseteq \chi A$ and a net $(\lambda_\alpha z_\alpha^*) \subseteq \mathbb{T} \Gamma_{z_0^*}$ so that $(t_\alpha)$ converges to $t_0$ (since $\chi A$ is dense in $\partial A$ \cite[Corollary 4.3.7]{Dales}) and $(\lambda_\alpha z_\alpha^*)$ converges weak-star to $z_0^*$, respectively, then $(\lambda_\alpha \delta_{t_\alpha} \otimes z_\alpha^*) $ converges weak-star to $\delta_{t_0} \otimes z_0^* = e^*$. It is clear by definition that $\sup_{x^* \in \Gamma_{e^*}} \rho (x^*) \leq \sup_{z^* \in \Gamma_{z_0^*}} \rho_Z (z^*) < 1$.

It remains to check that $\Gamma_{e^*}$ satisfies the property (iii). From this moment, the proof will follow the similar lines as the proof of \cite[Theorem 4.16]{CGKS2018}. 
Let $\eps >0$ and $U$ be a nonempty relatively $w^*$-open subset of $\Gamma_{e^*}$. Let $t_1 \in \chi A$ and $z_1^* \in \Gamma_{z_0^*}$ be such that $\delta_{t_1} \otimes z_1^* \in U$. Take $f_1, \ldots, f_n \in X$ such that $\delta_t \otimes z^* \in U$ whenever $\max_{1\leq k \leq n} | \langle \delta_t \otimes z^* - \delta_{t_1} \otimes z_1^*, f_k \rangle | < 1$. Consider the following sets
\begin{align*}
&B := \{ t \in \chi A : |z_1^* (f_k (t)) - z_1^* (f_k (t_1)) | < 1, \, k=1,\ldots, n\}, \\
&D := \{ z^* \in \Gamma_{z_0^*} : |z^* (f_k (t_1)) - z_1^* (f_k (t_1)) | < 1, \, k=1,\ldots, n\}.
\end{align*} 
Observe that $B$ is relatively open in $\chi A$ and $D$ is relatively $w^*$-open in $\Gamma_{z_0^*}$. Since the mapping 
\[
(t,z^*) \in \chi A \times (\Gamma_{z_0^*}, w^*) \longmapsto (z^* (f_1 (t)), \ldots, z^* (f_n (t)) ) \in \mathbb{K}^n 
\]
is continuous, there exist a nonempty open set $B_1 \subseteq B$ and a nonempty $w^*$-open set $D_1 \subseteq D$ such that for every $t \in B_1$ and $z^* \in D_1$, we have that $\max_{1\leq k \leq n} |z^* (f_k (t)) - z_1^* (f_k (t_1)) | < 1$. Define $W:= \{ \delta_t \otimes z^* : t \in B_1, z^* \in D_1\}$. It is clear that $W$ is contained in $U$. 

Now, applying the definition of quasi-$\ACK$ structure to $Z, \Gamma_Z, D_1$ and $\eps/2$, we get $V_Z \subseteq D_1, z_2^* \in V_Z, e_Z \in S_Z$ and $F_Z \in \mathcal{L}(Z)$ which satisfy the properties (i)'-(vi)'. Moreover, applying \cite[Lemma 4.4]{CGKS2018} to $A, \chi A, B_1$ and $\eps /2$, we may find a nonempty subset $B_2 \subseteq B_1$, functions $f_0, e_A \in A$ and $s_0 \in B_2$ satisfying its conclusion. 

Finally, define the nonempty subset $V \subseteq U$ and corresponding $x_1^* \in V, e\in S_X$ and $F \in \mathcal{L}(X)$ as follows: 
\begin{align*}
&V:= \{ \delta_t \otimes z^* : t \in B_2, z^* \in V_Z \} \subseteq W \subseteq U, \\
&x_1^* := \delta_{s_0} \otimes z_1^*, \quad e(t) := e_A (t) e_Z \,\, \text{ for every } t \in \partial A
\end{align*} 
and 
\[
(F(f))(t):= f_0 (t) F_Z (f(t)) \text{ for every } f \in X \text{ and } t \in \partial A.
\]
Arguing in the same way as in the proof of \cite[Theorem 4.16]{CGKS2018}, it can be verified that the set $V \subseteq U $ and elements $x_1^* \in V, e\in S_X, F \in \mathcal{L}(X)$ satisfy the conditions (i)'-(vi)'.
\end{proof} 

\begin{remark}
We do not know if the argument used in Theorem \ref{thm:unifalgebra} can be applied to the case when the space $C(\partial A, Z)$ is replaced by, for instance, $C(\partial A, (Z, w))$. The difficulties come from the lack of concrete representation of the extremal structure of the unit ball of $C(\partial A, (Z, w))^*$. As a matter of fact, its extreme points behave quite differently from that of the unit ball of $C(\partial A, Z)^*$ (see, for instance, \cite{HS} and the references therein). 
\end{remark}

\begin{cor}\label{cor:unifalgebra} Let $Z$ be a Banach space with quasi-$\ACK$ structure.
\begin{enumerate}
\setlength\itemsep{0.3em}
\item[\textup{(a)}] If $K$ is a compact Hausdorff space, then $C(K,Z)$ has quasi-$\ACK$ structure. 
\item[\textup{(b)}] If $\Omega$ is a completely regular Hausdorff space and if, in addition, $Z$ is finite dimensional, then $C_b (\Omega, Z)$ has quasi-$\ACK$ structure. 
\item[\textup{(c)}] $c_0 (Z)$ has quasi-$\ACK$ structure. 
\end{enumerate}
\end{cor} 

\begin{proof}
(a) is an obvious consequence of Theorem \ref{thm:unifalgebra}. For the item (b), note that $C_b (\Omega, Z)$ can be isometrically identified with $C(\beta \Omega, Z)$, where $\beta \Omega$ is the Stone-\v{C}ech compactification of $\Omega$. (c) follows from the observation that $c_0 (Z)$ can be viewed as a closed subspace of $C(\beta \mathbb{N}, Z)$. 
\end{proof}

\begin{remark}
We do not know if the class of Banach spaces with quasi-$\ACK$ structure is stable under $\ell_\infty$-sum operations in general. As far as we know, it is even unknown if property quasi-$\beta$ is stable under $\ell_\infty$-sums or not. 
Nevertheless, the assertion (b) of Corollary \ref{cor:unifalgebra} yields that $\ell_\infty(Z)$ has quasi-$\ACK$ structure whenever $Z$ is a {finite dimensional} Banach space with quasi-$\ACK$ structure. 
\end{remark} 

In fact, Corollary \ref{cor:unifalgebra} enables us to construct a non-trivial Banach space which is of quasi-$\ACK$ structure, but not $\ACK_\rho$ structure or property quasi-$\beta$. We first need the following technical lemma.

\begin{lemma}\label{lem:propB}
Let $K$ be a compact Hausdorff space and $Z$ be a Banach space. If $C(K,Z)$ has property \textup{B}, then $C(K)$ has property \textup{B}. 
\end{lemma}

\begin{proof}
Let $X$ be a Banach space and $T \in \Lin (X, C(K))$ with $\|T\|=1$. Fix $z_0 \in S_Z$ and define $\widetilde{T} \in \Lin (X, C(K,Z))$ by 
\[
\widetilde{T} (x)(t) := [(T(x))(t)] z_0 \quad (x \in X, \, t \in K). 
\]
By the assumption, given $\eps \in (0,1)$, there is an operator $\widetilde{S} \in \NA (X, C(K,Z))$ so that $\|\widetilde{S} -\widetilde{T} \| < \eps$ and $\|\widetilde{S}(x_0)\| = \|S\|=1$. Take $t_0 \in K$ and $z_1^* \in S_{Z^*}$ so that $\| (\widetilde{S}(x_0))(t_0) \| = \|\widetilde{S}( x_0)\|$ and $z_1^* ( ( \widetilde{S}(x_0) )(t_0) ) = |(T(x_0))(t_0)|^{-1} (T(x_0))(t_0)$ (noting that $(T(x_0))(t_0) \neq 0$). For simplicity, put $\theta := (T(x_0))(t_0)$ and observe that 
\begin{equation}\label{eq:theta}
| | \theta | z_1^* (z_0) -1 | = |z_1^* ((\widetilde{T} (x_0)) (t_0)) - z_1^* ((\widetilde{S}( x_0))(t_0)) | < \eps.
\end{equation}

Now, define ${S} \in \Lin (X, C(K))$ by 
\[
{S} (x)(t) := z_1^* ((\widetilde{S}(x))(t)) \quad (x\in X, \, t \in K). 
\]
Note that $\|{S}\| = \|{S} (x_0)\| = 1$. For $x \in B_X$, we have 
\begin{align*}
\| |\theta| {S} x - |\theta| z_1^* (z_0) T(x) \| = \sup_{t \in K} \bigl| |\theta| z_1^* ((\widetilde{S}(x))(t)) - |\theta| z_1^* (z_0) (T(x)) (t) \bigr| \leq \| \widetilde{S} - \widetilde{T} \| <\eps, 
\end{align*} 
which implies that $\| |\theta| {S} - |\theta| z_1^* (z_0) T \| \leq \eps$.
Note from \eqref{eq:theta} that $\| T - |\theta| z_1^* (z_0) T \| = |1 - |\theta| z_1^* (z_0) | \|T\|< \eps$, and $|\theta| {S} \in \NA(X, C(K))$; hence we complete the proof.
\end{proof}

\begin{prop}\label{prop:qACKbut}
There exists a Banach space $E$ satisfying that 
\begin{enumerate}
\setlength\itemsep{0.3em}
\item[\textup{(i)}] $E$ has quasi-$\ACK$ structure,
\item[\textup{(ii)}] $E$ fails to have property \textup{B}, 
\item[\textup{(iii)}] $(\ell_1^2, E)$ fails to have the \textup{BPBp} for operators.
\end{enumerate} 
Consequently, $E$ fails to have $\ACK_\rho$ structure for any $\rho \in [0,1)$ and property quasi-$\beta$. 
\end{prop} 

\begin{proof}
Let $K$ be a compact Hausdorff space so that $C(K)$ fails property \textup{B} (see \cite{JW2} or \cite{S}). As in Remark \ref{rem:basic}, let $Z$ be a Banach space satisfying property quasi-$\beta$ such that $(\ell_1^2, Z)$ fails to have the \textup{BPBp} for operators. It follows from \cite[Proposition 2.8]{ACKLM} that $(\ell_1^2, C(K,Z))$ cannot have the \textup{BPBp} for operators. Consequently, $E:= C(K,Z)$ fails to have $\ACK_\rho$ structure for any $\rho \in [0,1)$ by \cite[Theorem 3.4]{CGKS2018}, and $E$ fails to have property \textup{B} by Lemma \ref{lem:propB}. As the quasi-$\ACK$ structure of $E$ follows from Corollary \ref{cor:unifalgebra}.(a), we conclude that $E$ is the desired space. 
\end{proof}

By Corollary \ref{cor:denseness}, a Banach space satisfying the above (i)-(iii) in Proposition \ref{prop:qACKbut} has property \textup{B}$^k$. As far as we are aware, it was unknown whether such a Banach space $E$ constructed in the proof of Proposition \ref{prop:qACKbut} has property \textup{B}$^k$ or not. 

\subsection{Results on ACK structure}

In this section, we show that $\ACK$ structure of a Banach space $Z$ implies $\ACK$ structure of some spaces of $Z$-valued functions. To this end, we make use of the following lemma which can be viewed as a \emph{non-compact} version of \cite[Lemma 4.4]{CGKS2018}. In other words, we are going to consider a Urysohn-type lemma in the context of an algebra of functions on a (not necessarily compact) Hausdorff space. As an application, we observe $\ACK$ structure of certain spaces of bounded holomorphic functions between Banach spaces. 

Let us denote by $\rho A$ the set of all strong peak points of a subalgebra $A \subseteq C_b (\Omega)$ and by $St_\eps$ the \emph{Stolz's region} given by $St_\eps = \{z \in \mathbb{C}: |z| + (1-\eps)|1-z| \leq 1 \}$. 


\begin{lemma}\label{lem:Urysohn}
Let $\Omega$ be a Hausdorff space and $A$ be a subalgebra of $C_b (\Omega)$. Then, for every open subset $W$ of $\Omega$, $t_0 \in W \cap \rho A$ and $0 < \eps <1$, there exist a nonempty subset $W_0 \subseteq W$ and functions $f, e \in A$ such that $f(t_0) = \|f\| =1, e(t_0)=\|e\|=1$, $\sup_{t \in \Omega \setminus W_0} |f(t)|\leq \eps$, $|1-e(t)|<\eps$ for every $t \in W_0$ and $f(\Omega) \subseteq St_\eps$. 
\end{lemma} 

\begin{proof}
Apply \cite[Lemma 3]{KL} to get a function $e \in A$ such that $e(t_0)=\|e\| = 1$, $\sup_{t \in \Omega \setminus W} |e(t)|\leq \eps$ and $e(\Omega) \subseteq St_\eps$. Define $W_0 := \{ t \in W : |1-e(t)| < \eps\}$ and the function $f_n : \Omega \rightarrow \mathbb{K}$ by $f_n (t) = (e(t))^n$. Arguing as in the proof of \cite[Lemma 4.4]{CGKS2018}, we may find a suitable $n_0 \in \mathbb{N}$ so that $f:=f_{n_0}$ satisfies the conclusion of the lemma.
\end{proof} 

Recall that a linear topology $\tau$ on $Z$ is \emph{$\Gamma$-acceptable} for a $1$-norming set $\Gamma \subseteq B_{Z^*}$ if it is dominated by the norm topology and dominates $\sigma(Z,\Gamma)$. A function $f: \mathcal{T} \to \mathcal{S}$ between topological spaces $\mathcal{T}$ and $\mathcal{S}$ is \emph{quasi-continuous} if for every nonempty open set $U \subset \mathcal{T}$, $z \in U$ and any neighborhood $V \subseteq \mathcal{S}$ of $f(z)$, there exists a nonempty open subset $W \subseteq U$ such that $f(W) \subseteq V$. The next theorem is a general result on a vector-valued function spaces which can be seen as a non-compact version of \cite[Theorem 4.16]{CGKS2018}.

\begin{theorem}\label{thm:functionalgebra}
Let $\Omega$ be a Hausdorff space, $A$ be a subalgebra of $C_b(\Omega)$ and $Z$ be a Banach space with $\ACK_\rho$ structure \textup{(}resp., $\ACK$ structure\textup{)} witnessed by a $1$-norming set $\Gamma_Z$. Finally, let $\tau$ be a $\Gamma_Z$-acceptable topology on $Z$. Suppose that $\rho A$ is a $1$-norming set for $A$ and $X \subseteq C_b (\rho A, (Z,\tau))$ satisfies the following properties: 
\begin{enumerate}
\setlength\itemsep{0.3em}
\item[\textup{(i)}] For every $x \in X$ and $f \in A$, the function $t \mapsto f(t) x$ belongs to $X$.
\item[\textup{(ii)}] $X$ contains all functions of the form $f \otimes z$ for $f \in A$ and $z \in Z$. 
\item[\textup{(iii)}] $T\circ x \in X$ for every $x \in X$ and $T \in \mathcal{L} (Z)$.
\item[\textup{(iv)}] For every finite collection $\{ x_k\}_{k=1}^n \subseteq X$, the corresponding two-variable function $\phi : \rho A \times (\Gamma_Z, w^*)\rightarrow \mathbb{K}^n$, defined by $\phi(t, z^*)= (z^* (x_k (t)))_{k=1}^n$ is quasi-continuous. 
\end{enumerate} 
Then $X$ has $\ACK_\rho$ structure \textup{(}resp., $\ACK$ structure\textup{)}. 
\end{theorem}

\begin{proof}
We here give a proof of when $Z$ has $\ACK_\rho$ structure for some $\rho \in [0,1)$. Let $\Gamma := \{ \delta_t \otimes z^* : t \in \rho A, z^* \in \Gamma_Z \} \subseteq B_{Z^*}$. Then $\Gamma$ is $1$-norming for $X$. Fix $\eps >0$ and a nonempty relatively $w^*$-open subset $U$ of $\Gamma$. Let $t_0 \in \rho A$ and $z_0^* \in \Gamma_Z$ be such that $\delta_{t_0} \otimes z_0^* \in U$. Take $f_1, \ldots, f_n \in X$ as in the proof of Theorem \ref{thm:unifalgebra}, and consider the following sets
\begin{align*}
&B := \{ t \in \rho A : |z_0^* (f_k (t)) - z_0^* (f_k (t_0)) | < 1, \, k=1,\ldots, n\}, \\
&D := \{ z^* \in \Gamma_{Z} : |z^* (f_k (t_0)) - z_0^* (f_k (t_0)) | < 1, \, k=1,\ldots, n\}.
\end{align*} 
Observe that $B$ is relatively open in $\rho A$ (since $\tau$ is $\Gamma_Z$-acceptable) and $D$ is relatively $w^*$-open in $\Gamma_{Z}$. By the assumption (iv), there exist a nonempty open set $B_1 \subseteq B$ and a nonempty $w^*$-open set $D_1 \subseteq D$ such that for every $t \in B_1$ and $z^* \in D_1$, we have that $\max_{1\leq k \leq n} |z^* (f_k (t)) - z_0^* (f_k (t_0)) | < 1$. Define $W:= \{ \delta_t \otimes z^* : t \in B_1, z^* \in D_1\}$. 

Applying the definition of $\ACK_\rho$ structure to $Z, \Gamma_Z, D_1$ and $\eps/2$, we get $V_Z \subseteq D_1, z_2^* \in V_Z, e_Z \in S_Z$ and $F_Z \in \mathcal{L}(Z)$ which satisfy the properties (i)'-(vi)'. Moreover, applying Lemma \ref{lem:Urysohn} to $A, B_1$ and $\eps /2$, we may find a nonempty subset $B_2 \subseteq B_1$ and functions $f_0, e_A \in A$ satisfying its conclusion. 

Define the nonempty subset $V \subseteq U$ and corresponding $x_1^* \in V, e\in S_X, F \in \mathcal{L}(X)$ as in the proof of Theorem \ref{thm:unifalgebra}. For the same reason as before, we may conclude that these $V, x_1^*, e$ and $F$ satisfy the conditions (i)'-(vi)'. 
\end{proof}

Given complex Banach spaces $X$ and $Y$, let $\mathcal{H}^\infty (B_X^\circ, Y)$ be the Banach space of all bounded holomorphic functions from the open unit ball $B_X^\circ$ of $X$ into $Y$. We denote by $\mathcal{A}_u (B_X, Y)$ (resp., $\mathcal{A}_\infty (B_X, Y)$) the subspace of $\mathcal{H}^\infty (B_X^\circ, Y)$ of all members which are uniformly continuous on $B_X^\circ$ (resp., continuously extendable to $B_X$). 

Let $\mathcal{A} (B_X, Y)$ be either $\mathcal{A}_u (B_X, Y)$ or $\mathcal{A}_\infty (B_X, Y)$. It is well known that if $X$ is locally uniformly rotund (for short, \textup{LUR}), then every point in $S_X$ is a strong peak point for $\mathcal{A} (B_X, \mathbb{C})$. Moreover, it is observed in \cite{CLS} that if $X$ has the Radon-Nikod\'ym property (for short, \textup{RNP}), then $\rho \mathcal{A}(B_X, \mathbb{C})$ is a $1$-norming set for $\mathcal{A}(B_X, \mathbb{C})$.

\begin{cor}\label{cor:holo}
Let $X$ and $Y$ be complex Banach spaces. Suppose that $Y$ has $\ACK_\rho$ structure \textup{(}resp., $\ACK$ structure\textup{)}. If $X$ is \textup{LUR} or has the \textup{RNP}, then $\mathcal{A} (B_X, Y)$ has $\ACK_\rho$ structure \textup{(}resp., $\ACK$ structure\textup{)}. 
\end{cor} 

\begin{proof}
From the comment above, if we let $A := \mathcal{A} (B_X, \mathbb{C})$, then $\rho A$ is a norming set for $A \subseteq C_b (B_X)$ in any case. Notice that the restriction mapping from $\mathcal{A} (B_X, Y)$ to $\{ f \vert_{\rho A} : f \in \mathcal{A} (B_X, Y)\} \subseteq C_b (\rho A, Y)$ is an isometry. Applying Theorem \ref{thm:functionalgebra} to $\{ f \vert_{\rho A} : f \in \mathcal{A} (B_X, Y)\}$, we complete the proof.
\end{proof}

We finish this subsection by showing that the assumption on the domain $X$ in Corollary \ref{cor:holo} can be removed when the target space $Y$ is \emph{finite dimensional}. In order to do so, we borrow the notion of holomorphic functions whose range is relatively compact from \cite{Mujica}. Let us denote by $\mathcal{H}_K^\infty (B_X^\circ, Y)$ the subspace of all functions $f \in \mathcal{H}^\infty (B_X^\circ, Y)$ which have a relatively compact range. Let $\mathcal{A}_K (B_X, Y):= \mathcal{A} (B_X, Y) \cap \mathcal{H}_K^\infty (B_X^\circ, Y)$, where $\mathcal{A} (B_X, Y)$ stands for either $\mathcal{A}_u (B_X, Y)$ or $\mathcal{A}_\infty (B_X, Y)$. 

\begin{prop}\label{prop:cpt_holomorphic}
Let $X$ and $Y$ be complex Banach spaces. Suppose that $Y$ has $\ACK_\rho$ structure \textup{(}resp., $\ACK$ structure\textup{)} and $(Y,w)$ is Lindel\"of. Then $\mathcal{A}_K (B_X, Y)$ has $\ACK_\rho$ structure \textup{(}resp., $\ACK$ structure\textup{)}.
\end{prop} 

\begin{proof}
Given $f \in \mathcal{A}_K (B_X, Y)$, consider its Gelfand transform $\widehat{f} : \mathcal{M} \rightarrow Y^{**}$ given by $\widehat{f} (\phi) (y^*) := \phi (y^* f)$ for every $\phi \in \mathcal{M}$ and $y^* \in Y^*$, where $\mathcal{M}$ is the spectrum of $\mathcal{A} (B_X, \mathbb{C})$ endowed with the Gelfand topology. Notice that $\widehat{f} \in C(\mathcal{M}, (Y, w))$ and $\|\widehat{f} \| = \|f\|$ for each $f \in \mathcal{A}_K (B_X, Y)$. Indeed, as it is clear that $\widehat{f} \in C(\mathcal{M}, (Y^{**}, w^*))$, it suffices to check that $\widehat{f} (\phi)$ is actually a member of $Y$ for each $\phi \in \mathcal{M}$. Let $(y_\alpha^*) \subseteq B_{Y^*}$ be a net which converges weak-star to some $y_\infty^*$ in $B_{Y^*}$. Using the compactness of $\overline{f (B_X)}$, it follows that the net $(y_\alpha^* f)$ converges to $(y_\infty^* f)$ in $\mathcal{A} (B_X, \mathbb{C})$; hence the net $(\phi (y_\alpha^* f))$ converges to $\phi(y_\infty^* f)$. This shows that $\widehat{f} (\phi)$ is weak-star continuous on $Y^*$ \cite[Corollary 2.7.9]{megginson}; hence $\widehat{f}(\phi) \in Y$. 

Let us say $E = \{ \widehat{f} : f \in \mathcal{A}_K (B_X, Y)\} \subseteq C(\mathcal{M}, (Y,w))$. Define $A:= \{ y^* \widehat{f} : y^* \in Y^*, \widehat{f} \in E \} \subseteq C(\mathcal{M})$. Note that $A$ is a uniform algebra on $\mathcal{M}$. Also, $A$ and $E$ are isometrically isomorphic to $\{ h \vert_{\partial A} : h \in A\} \subseteq C(\partial A)$ and $\{ \widehat{f} \vert_{\partial A} : \widehat{f} \in E \} \subseteq C(\partial A, (Y,w))$, respectively. To complete the proof, it suffices to claim that $E$ satisfies the conditions given in \cite[Theorem 4.16]{CGKS2018}. It is straightforward to check that $(y^* \widehat{f} ) \widehat{g} = \widehat{ (y^* f) g } \in E$ for every $f, g \in \mathcal{A}_K (B_X, Y)$ and $y^* \in Y^*$. Moreover, $(y^* \widehat{f} )\otimes y$ belongs to $E$ for each $y^* \widehat{f} \in A$ and $y \in Y$ since it is the same as $\widehat{g}$, where $g \in \mathcal{A}_K (B_X, Y)$ is given by $g(x) = y^* (f(x)) y$ for every $x \in B_X$. Note also that $T \widehat{f} = \widehat{T f}$ for every $T \in \Lin (Y)$ and $f \in \mathcal{A}_K (B_X, Y)$. Finally, since $(Y,w)$ is Lindel\"of, we obtain that $\widehat{f} ( \mathcal{M})$ is Lindel\"of in $(Y,w)$; hence $\widehat{f} ( \mathcal{M})$ is norm-fragmented \cite[Corollary E]{CNV}. Thus, \cite[Proposition 4.21]{CGKS2018} completes the proof of the claim. 
\end{proof}

The following promised result is an immediate corollary of Proposition \ref{prop:cpt_holomorphic}.

\begin{cor}
Let $X$ be a complex Banach space and $Y$ be a finite dimensional space. If $Y$ has $\ACK_\rho$ structure \textup{(}resp., $\ACK$ structure\textup{)}, then $\mathcal{A} (B_X, Y)$ has $\ACK_\rho$ structure \textup{(}resp., $\ACK$ structure\textup{)}. 
\end{cor} 

\section{On the stability of quasi-\textup{ACK} structure}

\subsection{Injective tensor products}

Recall that for Banach spaces $X$ and $Y$, the \emph{injective tensor product} of $X$ and $Y$, denoted by $X \iten Y$, is the completion of $X \otimes Y$ under the norm given by 
\[
\|u\| = \sup \left\{ \left| \sum_{i=1}^n x^* (x_i) y^* (y_i) \right| : x^* \in B_{X^*}, y^* \in B_{Y^*} \right\}, 
\]
where $\sum_{i=1}^n x_i \otimes y_i$ is any representation of $u$. We refer the reader to \cite{Ryan} for background on tensor products of Banach spaces. 
We begin this subsection by observing that quasi-$\ACK$ structure is stable under injective tensor products, which is an analogue of \cite[Theorem 4.12]{CGKS2018} for quasi-$\ACK$ structures. 

\begin{theorem}\label{thm:iten}
Let $X$ and $Y$ be Banach spaces with quasi-$\ACK$ structure. Then $X \widehat{\otimes}_\eps Y$ has quasi-$\ACK$ structure. 
\end{theorem} 

\begin{proof}
Take a $1$-norming subset $\Gamma_X \subseteq S_{X^*}$, $\Gamma_Y \subseteq S_{Y^*}$ and functions $\rho_X : \Gamma_X \rightarrow [0,1)$, $\rho_Y : \Gamma_Y \rightarrow [0,1)$ accordingly following the Definition \ref{def:qACK}. Define the map 
\[
\phi : (B_{X^*}, w^*) \times (B_{Y^*}, w^*) \longrightarrow (B_{(X \widehat{\otimes}_\eps Y)^*}, w^*) 
\] 
by $\phi (x^*, y^*) = x^* \otimes y^*$ for every $x^* \in B_{X^*}$ and $y^* \in B_{Y^*}$. Then $\phi$ is continuous. Consider the set $\Gamma := \phi (\Gamma_X \times \Gamma_Y)$ and define $\rho : \Gamma \rightarrow [0,1)$ by $\rho ( \phi (x^*, y^*)) = \max \{ \rho_X (x^*), \rho_Y (y^*)\}$. 

To see that $\Gamma$ and $\rho$ are the desired ones, let $e^* \in \operatorname{Ext} ( B_{(X \widehat{\otimes}_\eps Y)^*} )$ be given. By Tseitlin \cite{T} or Ruess-Stegall \cite{RS1982}, $\operatorname{Ext} ( B_{(X \widehat{\otimes}_\eps Y)^*} ) = \operatorname{Ext} (B_{X^*}) \otimes \operatorname{Ext} (B_{Y^*})$; hence $e^* = e_X^* \otimes e_Y^*$ for some $e_X^* \in \operatorname{Ext} (B_{X^*})$ and $e_Y^* \in \operatorname{Ext} (B_{Y^*})$. Thus, we can find $\Gamma_{e_X^*} \subseteq \Gamma_X$ and $\Gamma_{e_Y^*} \subseteq \Gamma_Y$ satisfying the conditions (i)-(iii) in Definition \ref{def:qACK}. Note that if nets $(x_\alpha^*) \subseteq \Gamma_{e_X^*}$ and $(y_\beta^*) \subseteq \Gamma_{e_Y^*}$ satisfy that $(c_\alpha x_\alpha^*)$ and $(d_\beta y_\beta^*)$ converge weak-star to $e_X^*$ and $e_Y^*$ for some $(c_\alpha)$ and $(d_\beta) \subseteq \mathbb{T}$, respectively, then $c_\alpha d_\beta \phi (x_\alpha^*, y_\beta^*)$ converges weak-star to $e^*$. This shows that $e^* \in \overline{\mathbb{T} \phi (\Gamma_{e_X^*} \times \Gamma_{e_Y^*} ) }^{w^*}$. Putting $\Gamma_{e^*} := \phi (\Gamma_{e_X^*} \times \Gamma_{e_Y^*} )$, we have that 
\[
\sup \{ \rho (z^*) : z^* \in \Gamma_{e^*} \} \leq \max \bigl\{ \sup \{ \rho_{X} (x^*) : x^* \in \Gamma_{e_X^*}\},\, \sup \{ \rho_{Y} (y^*) : y^* \in \Gamma_{e_Y^*}\}\bigr\} < 1. 
\]

It remains to check that $\Gamma_{e^*}$ satisfies the condition (iii). Let $\eps >0$ and a non-empty relatively $w^*$-open subset $U \subseteq \Gamma_{e^*}$ be given. Pick $x_0^* \in \Gamma_{e_X^*}$ and $y_0^* \in \Gamma_{e_Y^*}$ so that $\phi (x_0^*, y_0^*) \in U$. Using the continuity of $\phi$, find relatively $w^*$-open neighborhoods $W_X \subseteq \Gamma_{e_X^*}$ and $W_Y \subseteq \Gamma_{e_Y^*}$ of $x_0^*$ and $y_0^*$, respectively, such that $\phi(W_X\times W_Y) \subseteq U$. Now, applying the definition of quasi-$\ACK$ to $X$ and $Y$ with respect to $\eps/2$, $W_X$ and $W_Y$, we obtain 
\begin{align*}
&V_X \subseteq W_X, x_1^* \in V_X, e_X \in S_X \text{ and } F_X \in \mathcal{L} (X), \\
&V_Y \subseteq W_Y, y_1^* \in V_Y, e_Y \in S_Y \text{ and } F_Y \in \mathcal{L} (Y). 
\end{align*} 
Define the set $V \subseteq U$, $z_1^* \in V$, $e \in S_{X \iten Y}$ and $F \in \mathcal{L} (X \iten Y)$ as follows: $V:= \phi (V_X \times V_Y)$, $z_1^* := \phi(x_1^*, y_1^*)$, $e:=e_X \otimes e_Y$ and $F:=F_X \otimes F_Y$. Arguing as in the proof of \cite[Theorem 4.12]{CGKS2018}, we can conclude that $V, z_1^*, e$ and $F$ satisfy (i)'-(vi)'. 
\end{proof}

It is well known that $C(K,Z) = C(K) \iten Z$ for a compact Hausdorff space $K$ and a Banach space $Z$. Thus, we have the following immediate corollary which was already covered by Corollary \ref{cor:unifalgebra}.

\begin{cor}
Let $K$ be a compact Hausdorff space. If $Z$ has quasi-$\ACK$ structure, then $C(K,Z)$ has quasi-$\ACK$ structure.
\end{cor}

We checked in Remark \ref{rem:basic} that a Banach space with property quasi-$\beta$ has quasi-$\ACK$ structure, thus Theorem \ref{thm:iten} can be applied to such a space. However, the next result shows that property quasi-$\beta$ (or property $\beta$) is actually stable under injective tensor products. Up to our knowledge, this result has not appeared in the literature. 

\begin{prop}\label{prop:beta-iten}
Let $X$ and $Y$ be Banach spaces.
\begin{enumerate}
\setlength\itemsep{0.3em}
\item[\textup{(a)}] If $X$ and $Y$ have property $\beta$, then $X \iten Y$ has property $\beta$,
\item[\textup{(b)}] If $X$ and $Y$ have property quasi-$\beta$, then $X \iten Y$ has property quasi-$\beta$.
\end{enumerate} 
\end{prop}

\begin{proof}
(a). Let $(x_\alpha, x_\alpha^*) \subseteq S_X \times S_{X^*}$ and $(y_\beta, y_\beta^*) \subseteq S_Y \times S_{Y^*}$ be pairs of vectors and $\rho_1,\rho_2 \in [0,1)$ to be the constants coming from the definition of property $\beta$ of $X$ and $Y$, respectively. Consider the set $\{ (x_\alpha \otimes y_\beta, x_\alpha^* \otimes y_\beta^*) \}$. Note that $|(x_\alpha^* \otimes y_\beta^*) (x_{\alpha'} \otimes y_{\beta'})| \leq \rho:= \max\{\rho_1,\rho_2\} \in [0,1)$ whenever $(\alpha,\beta)\neq(\alpha',\beta')$. Thus, it remains to check that $\{ x_\alpha^* \otimes y_\beta^* \}$ is a norming set. Let $u \in X \iten Y$ and view $u$ as an element of $K_{w^*} (X^*, Y)$, the space of $w^*$-$w$-continuous compact linear operators. Since $B_{X^*} = \overline{\text{aco}} \{x_\alpha^*\}$, we can find $x_\alpha^*$ such that $\| u(x_\alpha^*)\|$ is close to $\|u\|$. Again, since $B_{Y^*} = \overline{\text{aco}} \{y_\beta^*\}$, there is $y_\beta^*$ such that $|y_\beta^* (u (x_\alpha^*))|$ is close to $\|u(x_\alpha^*)\|$. This finishes the proof. 

(b). Let $A_X = \{x_\lambda^* : \lambda \in \Lambda_X\} \subseteq S_{X^*}$, $\{x_\lambda: \lambda \in \Lambda_X\}\subseteq S_X$, $A_Y = \{y_\lambda^* : \lambda \in \Lambda_Y \} \subseteq S_{Y^*}$, $\{y_\lambda: \lambda \in \Lambda_Y\} \subseteq S_Y$, and $\rho_X : A_X \rightarrow \mathbb{R}$, $\rho_Y : A_Y \rightarrow \mathbb{R}$ be the objects witnessing property quasi-$\beta$ of $X$ and $Y$. Define $A := A_X \otimes A_Y$, consider $\{x_\lambda: \lambda \in \Lambda_X\} \otimes \{y_\lambda: \lambda \in \Lambda_Y\}$ and $\rho (x^* \otimes y^*) = \max \{ \rho_X (x^*), \rho_Y (y^*) \}$ for every $x^* \otimes y^* \in A$. Arguing as in Theorem \ref{thm:iten}, we can observe that for every $e^* \in \Ext (B_{(X \iten Y)^*})$, there exist $A_{e_X^*} \subseteq A_X$ and $A_{e_Y^*} \subseteq A_Y$ such that $e^* \in \overline{ \mathbb{T} A_{e_X^*} \otimes A_{e_Y^*}}^{w^*}$ and $\sup \{ \rho (z^*) : z^* \in A_{e_X^*} \otimes A_{e_Y^*} \} < 1$. 
\end{proof}

It is worth mentioning that, up to our knowledge, it is not known whether property \textup{B} and property \textup{B}$^k$ are stable under taking injective tensor products. Nevertheless, we observe that property \textup{B}$^k$ of a reflexive Banach space is inherited to the injective tensor product between itself and an $L_1$-predual space.

\begin{theorem}
Let $X$ be an $L_1$-predual space and $Y$ be a reflexive Banach space with property \textup{B}$^k$, then $X \iten Y$ has property \textup{B}$^k$. 
\end{theorem}

\begin{proof}
Let $E$ be a Banach space, $T \in \mathcal{K} (E, X \iten Y)$ and $\eps >0$ be given. Consider the bilinear mapping $B_T : E \times Y^* \rightarrow X$ defined as
\[
B_T (e, y^*) (x^*) = \langle (T(e)) (x^*), y^* \rangle \,\, \text{ for every } (e,y^*) \in E \times Y^* \text{ and } x^* \in X^*, 
\]
noting that $\langle (T(e)) ( \cdot ), y^* \rangle$ is an element of $X$, where $T(e)$ is viewed as a member of $\mathcal{K}_{w^*} (X^*, Y)$. As a matter of fact, we remark that $X \iten Y \stackrel{1}{=} \mathcal{K}_{w^*} (X^*, Y)$ since $X$ is an $L_1$-predual; hence it has the approximation property (see \cite[Corollary 4.13]{Ryan}). 
Denote by $\widetilde{T} \in \Lin ( E \pten Y^*, X)$ the linearization of $B_T$, and observe that $\widetilde{T}$ is compact. Indeed, the relatively compactness of the set $\{ \widetilde{T} (e \otimes y^*) : e \in B_E, y^* \in B_{Y^*} \}$ is obtained from the compactness of $T$ and $w^*$-compactness of $B_{Y^*}$. 

Take a net $(Q_\lambda)$ of norm-one projections on $X$ converging to the identity on $X$ in the strong operator topology such that $Q_\lambda (X) \stackrel{1}{=} \ell_\infty^{n_\lambda}$ for some $n_\lambda \in \mathbb{N}$ (see, for instance, \cite[Chapter 7]{Lacey}). From the compactness of $\widetilde{T}$, we can find $\lambda_0$ so that $\| Q_{\lambda_0} \widetilde{T} - \widetilde{T} \| < \eps$. Put $Q = Q_{\lambda_0}$ and $n = n_{\lambda_0}$ for simplicity.

Consider the operator $(Q \widetilde{T})^* : Q(X)^* \rightarrow (E\pten Y^*)^* \stackrel{1}{=} \Lin (E, Y)$. We claim that the range of $(Q \widetilde{T})^*$ is contained in $\mathcal{K} (E,Y)$. To this end, let $z^* \in Q(X)^*$ and a sequence $(e_k) \subseteq B_E$ be fixed. Passing to a subsequence, we may assume that $(T(e_k))$ converges in norm to some $u \in X \iten Y =  \mathcal{K}_{w^*} (X^*, Y)$. Then for each $y^* \in Y^*$ we have
\begin{align*}
|[ (Q\widetilde{T})^* (z^*) ] (e_k \otimes y^*) - \langle u^* (y^*), Q^* (z^*) \rangle | 
&= |\langle (T(e_k))^* (y^*) - u^*(y^*) , Q^* (z^*) \rangle | \\
&\leq \| (T(e_k))^* - u^* \| \| y^* \| \| Q^* (z^*) \| \longrightarrow 0
\end{align*} 
as $k$ tends to $\infty$. This implies that 
\[
\| [ (Q\widetilde{T})^* (z^*) ] (e_k \otimes \cdot ) - \langle u^* (\cdot), Q^* (z^*) \rangle \| \longrightarrow 0;
\]
hence $(Q\widetilde{T})^* (z^*)$ belongs to $\mathcal{K} (E,Y)$.

Next, consider the bilinear mapping $S : E \times Q(X)^* \rightarrow Y$ defined as 
\[
S(e,z^*) = [ (Q\widetilde{T})^* (z^*) ] (e \otimes \cdot ) \,\, \text{ for every } (e,z^*) \in E \times Q(X)^*. 
\]
Since $Q(X)^*$ is finite dimensional and the range of $(Q\widetilde{T})^*$ is contained in $\mathcal{K}(E,Y)$, the linearization $\widetilde{S} \in \Lin (E \pten Q(X)^*, Y)$ of $S$ is compact. By the assumption, there exists $R \in \mathcal{K} (E \pten Q(X)^*, Y)$ such that $\| R - \widetilde{S} \| < \eps$ and $R$ attains its norm. Observe the identification $E \pten Q(X)^* \stackrel{1}{=} \ell_1^n (E)$ and let us denote by $(u_i^*)_{i=1}^n$ the basis for $Q(X)^*$ isometrically equivalent to the canonical basis for $\ell_1^n$. Let $\mu = \sum_{i=1}^n \lambda_i \xi_i \otimes u_i^*$, where $(\xi_i)_{i=1}^n \subseteq S_E$ and $(\lambda_i)_{i=1}^n \subseteq \mathbb{K}$, such that $\|R(\mu)\| = \|R\|$ and $\|\mu\| = \sum_{i=1}^n |\lambda| = 1$. This implies that there exists at least one $i_0 \in \{1,\ldots, n\}$ such that $\| R(\xi_{i_0}\otimes u_{i_0}^*) \| = \|R\|$. 

Finally, let us consider $\overline{R} : E \rightarrow \mathcal{K}_{w^*} (X^*, Y) \stackrel{1}{=} X \iten Y$ defined as 
\[
(\overline{R} (e)) (x^*) = R(e \otimes x^* \vert_{Q(X)}) \,\, \text{ for every } e \in E \text{ and } x^* \in X^*.
\] 
It is clear that $\| \overline{R} \| = \|\overline{R} (\xi_{i_0})\|$. Moreover, for every $e \in B_E$, 
\begin{align*}
\| \overline{R} (e) - T(e) \| &= \sup_{x^* \in B_{X^*} } \| R(e\otimes x^* \vert_{Q(X)}) - \langle \widetilde{T} (e \otimes \cdot), x^* \rangle \| \\
&\leq \sup_{x^* \in B_{X^*} } \| R(e\otimes x^* \vert_{Q(X)}) - \langle Q\widetilde{T} (e \otimes \cdot), x^* \vert_{Q(X)} \rangle \| + \eps \\
&= \sup_{x^* \in B_{X^*} } \| R(e\otimes x^* \vert_{Q(X)}) - S(e, x^* \vert_{Q(X)}) \|  + \eps \leq 2\eps; 
\end{align*} 
hence $\|\overline{R} - T \| \leq 2\eps$. It remains to prove that $\overline{R}$ is compact. Let $(e_k) \subseteq B_E$ be fixed. Passing to a subsequence, we may assume that $\sup_{1\leq i \leq n} \| R(e_k \otimes u_i^*) - y_i \| \rightarrow 0$ as $k \rightarrow \infty$, for some $y_1,\ldots, y_n \in Y$. Define the operator $G \in \mathcal{K}_{w^*} (X^*, Y)$ by $G(x^*) = \sum_{i=1}^n x^* (u_i) y_i$, where $(u_i)_{i=1}^n \subseteq Q(X)$ are the coordinate functionals to $(u_i^*)_{i=1}^n$. Observe that $(\overline{R} (e_k))$ converges to $G$ as $k \rightarrow \infty$, which implies that $\overline{R}$ is compact.
\end{proof} 

\begin{cor}
Let $Y$ be a reflexive Banach space. If $Y$ has property \textup{B}$^k$, then so does $C(K,Y)$. 
\end{cor} 

\subsection{$c_0$- or $\ell_\infty$-sum} 

We observed in Corollary \ref{cor:unifalgebra} that $c_0 (Z)$ has quasi-$\ACK$ structure whenever a Banach space $Z$ has quasi-$\ACK$ structure. First, we observe in a direct way that this result can be extended to a general $c_0$-sum of arbitrary Banach spaces with quasi-$\ACK$ structure. We start with the following well known result (see, for instance, \cite[Lemma 1.5]{HWW}). 

\begin{lemma}\label{lem:ext_ell1}
Let $I$ be an arbitrary set and $\{X_i\}_{i \in I}$ be a family of Banach spaces for each $i \in I$. Then 
\[
\Ext \left(B_{(\oplus_{i\in I} X_i)_{\ell_1}} \right) = \bigcup_{i \in I} \Ext (B_{X_i}), 
\]
where $\Ext (B_{X_i})$ is understood as a subset of $(\oplus_{i\in I} X_i)_{\ell_1}$. 
\end{lemma} 

\begin{theorem}\label{thm:c0sum}
Let $I$ be an arbitrary set and $\{X_i\}_{i \in I}$ be a family of Banach spaces with quasi-$\ACK$ structure for each $i \in I$. Then $Z := ( \oplus_{i \in I} X_i )_{c_0}$ has quasi-$\ACK$ structure. 
\end{theorem}

\begin{proof}
For each $i \in I$, let us denote by $P_i$ and $J_i$ the canonical projection from $Z$ to $X_i$ and the canonical injection from $X_i$ to $Z$, respectively. Let $\Gamma_i \subseteq B_{X_i^*}$ and $\rho_i : \Gamma_i \rightarrow [0,1)$ be the ones that witness quasi-$\ACK$ structure of $X_i$ for each $i \in I$. Define the set $\Gamma \subseteq B_{Z^*}$ and the function $\rho : \Gamma \rightarrow [0,1)$ as follows:
\begin{align*}
&\Gamma := \bigcup_{i\in I} P_i^* (\Gamma_i), \\
&\rho (P_i^* (x^*)) := \rho_i (x^*) \,\, \text{ for each } i \in I \text{ and } x^* \in \Gamma_i. 
\end{align*} 
Then it is clear that $\Gamma$ is a $1$-norming set. For a fixed $e^* \in B_{Z^*}$, by Lemma \ref{lem:ext_ell1}, there exists $i_0 \in I$ such that $e^* = P_{i_0}^* (e_{i_0}^*)$ for some $e_{i_0}^* \in \Ext (B_{X_{i_0}^*})$. Take a subset $\Gamma_{e_{i_0}^*} \subseteq \Gamma_{i_0}$ so that $e_{i_0}^* \in \overline{\mathbb{T} \Gamma_{e_{i_0}^*} }^{w^*}$. Then it is clear that $e^* \in \overline{ \mathbb{T} P_{i_0}^* (\Gamma_{e_{i_0}^*}) }^{w^*}$. Set $\Gamma_{e^*} := P_{i_0}^* (\Gamma_{e_{i_0}^*}) \subseteq \Gamma$ and note that 
\[
\sup \{ \rho (u^*) : u^* \in \Gamma_{e^*} \} \leq \sup \{ \rho_{i_0} (x^*) : x^* \in \Gamma_{e_{i_0}^*} \} < 1. 
\]

Now, let $\eps >0$ and $U$ be a nonempty relatively $w^*$-open subset of $\Gamma_{e^*}$. Note that $J_{i_0}^* (U)$ is a nonempty relatively $w^*$-open subset of $J_{i_0}^* (\Gamma_{e^*}) = \Gamma_{e_{i_0}^*}$. Thus, we can find a nonempty subset $V_{i_0} \subseteq J_{i_0}^* (U), x_{i_0}^* \in V_{i_0}, e_{i_0} \in S_{X_{i_0}}$ and $F_{i_0} \in \mathcal{L} (X_{i_0})$ and our claim is to show that
\begin{align*}
&V:= P_{i_0}^* (V_{i_0}), \,\, z_1^* := P_{i_0}^* (x_{i_0}^*), \,\, e:=J_{i_0} (e_{i_0}) \,\, \text{ and } \,\, F:= J_{i_0} F_{i_0} P_{i_0}
\end{align*}
are the desired elements corresponding to quasi-$\ACK$ structure of $Z$. Observe that $V$ is nonempty since $V_{i_0}$ is nonempty and elements in $P_{i_0}^*(V_{i_0})$ are linear extensions of elements in $V_{i_0}$ to $Z$. 

It is clear that $\|F\| = \|F(e)\| =1$, $z_1^* (F(e)) = x_{i_0}^* (F_{i_0} (e_{i_0}) ) =1$ and $F^* (z_1^*) = z_1^*$. Moreover, if $x^* \in V_{i_0, 1}$ (the subset of $\Gamma_{i_0}$ which appears in the property (iv)'), then 
\begin{align*}
\| F^* ( P_{i_0}^* (x^*)) \| &+ (1-\eps) \| (I_{Z^*} - F^*) (P_{i_0}^* (x^*)) \| \\
&= \| F_{i_0}^* (x^*) \| + (1-\eps) \|( I_{X_{i_0}^*} - F_{i_0}^* )(x^*) \| \leq 1,
\end{align*} 
which implies that $P_{i_0}^* (x^*) \in V_1 := \{ z^* \in \Gamma : \|F^*(z^*)\| + (1-\eps) \|z^* - F^*(z^*)\| \leq 1 \}$. Thus, if $z^* \in \Gamma \setminus V_1$, then $z^* = P_{i}^*( x_i^*)$ for some $x_i^* \in \Gamma_i$ with $i \neq i_0$, or $z^* = E_{i_0} (u_{i_0}^*)$ with $u_{i_0}^* \notin V_{i_0,1}$. In any case, 
\[
|z^* (F(e))| \leq \rho_{i_0} (x_{i_0}^*) = \rho(z_1^*).
\]
It is immediate to check the properties (v)' and (vi)'; so the proof is finished. 
\end{proof}

\begin{remark}
The absence of an explicit representation of the extremal structure on the unit ball of the dual of $Z := ( \oplus_{i \in I} X_i )_{\ell_\infty}$ makes an investigation of quasi-$\ACK$ structure of $Z$ difficult. We do not know whether $Z$ has quasi-$\ACK$ structure when each $X_i$ has quasi-$\ACK$ structure for each $i \in I$. 
\end{remark}

A very similar proof (but simpler) to the one of Theorem \ref{thm:c0sum} shows the following, which improves \cite[Theorem 4.11 and Corollary 4.19]{CGKS2018} at once. 

\begin{prop}\label{prop:intermediate}
Let $I$ be an arbitrary set and $\{X_i\}_{i \in I}$ be a family of Banach spaces with $\ACK_{\rho_i}$ structure for each $i \in I$. Suppose that $\rho := \sup_{i \in I} \rho_i < 1$. If $Z$ is a Banach space satisfying $( \oplus_{i \in I} X_i )_{c_0} \subseteq Z \subseteq ( \oplus_{i \in I} X_i )_{\ell_\infty}$, then $Z$ has $\ACK_\rho$ structure. 
\end{prop}

\section*{Acknowledgment}

The authors would like to thank Sun Kwang Kim and Miguel Mart\'in for valuable comments and remarks leading to improvement of this paper. The authors also want to thank anonymous referees for their careful reading and helpful suggestions.

The first author was supported by Basic Science Research Program through the National Research Foundation of Korea (NRF) funded by the Ministry of Education (2022R1A6A3A01086079) and by the Ministry of Education, Science and Technology [NRF-2020R1A2C1A01010377]. The second author was supported by NRF [NRF-2019R1A2C1003857], by POSTECH Basic Science Research Institute Grant [NRF-2021R1A6A1A10042944] and by a KIAS Individual Grant [MG086601] at Korea Institute for Advanced Study.

\end{document}